\documentclass[11pt,reqno,tbtags,a4paper]{amsart}
\usepackage{amssymb}
\usepackage{xpunctuate}
\usepackage{url}
\usepackage[square,numbers]{natbib}
\bibpunct[, ]{[}{]}{;}{n}{,}{,}

\title[On the Skorohod topology for a completely regular space]
{On the Skorohod topology for functions with values in a 
 completely regular space} 

\date{12 November, 2025}
\author{Svante Janson}
\thanks{Supported by the Knut and Alice Wallenberg Foundation
and
the Swedish Research Council% (VR)
}
\address{Department of Mathematics, Uppsala University, PO Box 480,
SE-751~06 Uppsala, Sweden}
\email{svante.janson@math.uu.se}
\newcommand\urladdrx[1]{{\urladdr{\def~{{\tiny$\sim$}}#1}}}
\urladdrx{http://www2.math.uu.se/~svantejs}
\numberwithin{equation}{section}

\renewcommand\le{\leqslant}

\allowdisplaybreaks

\theoremstyle{plain}% default
\newtheorem{theorem}{Theorem}[section]
\newtheorem{lemma}[theorem]{Lemma}

\newtheorem{corollary}[theorem]{Corollary}

\theoremstyle{definition}

\newcommand\xqed[1]{%
    \leavevmode\unskip\penalty9999 \hbox{}\nobreak\hfill
    \quad\hbox{#1}}

\newtheorem{exampleqqq}[theorem]{Example}
\newenvironment{example}{\begin{exampleqqq}}
  {\xqed{$\triangle$}\end{exampleqqq}}
\newtheorem{remarkqqq}[theorem]{Remark}
\newenvironment{remark}{\begin{remarkqqq}}
  {\xqed{$\triangle$}\end{remarkqqq}}
\newtheorem*{ack}{Acknowledgement}

\theoremstyle{remark}

\newcounter{dummy}
\makeatletter
\newcommand\myitem[1][]{\item[#1]\refstepcounter{dummy}\def\@currentlabel{#1}}
\makeatother

\newenvironment{romenumerate}[1][-10pt]{% optional argument changes indentation
\addtolength{\leftmargini}{#1}\begin{enumerate}% gives (i), (ii) etc.
 \renewcommand{\labelenumi}{\textup{(\roman{enumi})}}%
 }{\end{enumerate}}

\newcounter{oldenumi}
{\setcounter{oldenumi}{\value{enumi}}
\begin{romenumerate} \setcounter{enumi}{\value{oldenumi}}}
{\end{romenumerate}}

\newcounter{thmenumerate}

\newcounter{xenumerate}   %no left indentation; thus wider lines

\newcommand{\refT}[1]{Theorem~\ref{#1}}
\newcommand{\refTs}[1]{Theorems~\ref{#1}}
\newcommand{\refC}[1]{Corollary~\ref{#1}}

\newcommand{\refL}[1]{Lemma~\ref{#1}}

\newcommand{\refS}[1]{Section~\ref{#1}}

\begingroup
  \divide\count255 by 60
  \multiply\count255 by -60
  \advance\count255 by \time
  \ifnum \count255 < 10 \xdef\klockan{\the\count1.0\the\count255}
  \else\xdef\klockan{\the\count1.\the\count255}\fi
\endgroup

\newcommand\set[1]{\ensuremath{\{#1\}}}

\newcommand\xpar[1]{(#1)}
\newcommand\bigpar[1]{\bigl(#1\bigr)}

\newcommand\Bigcpar[1]{\Bigl\{#1\Bigr\}}

\def\rompar(#1){\textup(#1\textup)}    % usage: \rompar(...)

\def\xexp(#1){e^{#1}}

\newcommand\ntoo{\ensuremath{{n\to\infty}}}

\newcommand\downto{\searrow}
\newcommand\upto{\nearrow}
\newcommand\punkt{\xperiod}    % xpunctuate

\newcommand\eg{e.g\punkt}

\newcommand\bbR{\mathbb R}
\newcommand\bbC{\mathbb C}
\newcommand\bbN{\mathbb N}

\newcounter{CC}
\newcounter{cc}
\newcommand\gd{\delta}

\newcommand\gl{\lambda}
\newcommand\gL{\Lambda}

\newcommand\gs{\sigma}

\newcommand\eps{\varepsilon}
\renewcommand\phi{\xxx}  %% WARNING

\newcommand\oi{\ensuremath{[0,1]}}
\newcommand\ooi{(0,1]}
\newcommand\oio{\ensuremath{[0,1)}}
\newcommand\ooo{[0,\infty)}

\newcommand\hI{\widehat I}
\newcommand\td{\tilde d}
\newcommand\tzeta{\widetilde\zeta}
\newcommand\xzeta{\zeta^*}
\newcommand\txzeta{\widetilde\zeta^*}
\newcommand\cadlag{c\`adl\`ag}

\begin{document}

\begin{abstract} 
We correct a gap in the proof of a basic theorem by Jakubowski (1986)
on the Skorohod topology on the space of functions on $[0,1]$ with values in
a completely regular topological space.
\end{abstract}

\maketitle

\section{Introduction}\label{S:intro}

The Skorohod ($J_1$) topology 
on the space
$D(\oi)$ of functions on $\oi$ that are 
\cadlag{}
(i.e., right-continuous on $\oi$ and with left limits at all $t\in\ooi$)
is of fundamental importance in the study of stochastic processes.
The topology was introduced by \citet{Skorohod} for real-valued
functions; this is perhaps still the most important case, but 
it has also been very useful to extend the definition to functions on $\oi$
with values in other spaces. In particular, the extension to functions with
values in a  metric space is straightforward, 
see \eg{} \cite{EthierKurtz}. % and \cite{Billingsley2ed}.
%(The metric space is often assumed to be complete and separable.)

A further extension to functions with values in an arbitrary completely
regular topological space $E$ was made by \citet{Jakubowski},
generalizing a special case by \citet{Mitoma}.
Unfortunately there is a gap in the proof of the basic theorem 
\cite[Theorem 1.3]{Jakubowski} 
showing that the constructed topology depends only on the topology of $E$
(and not on the pseudometrics used in the construction, see \refS{S3} below).
It is easy to give a complete proof, but since we have not been able to find
a published proof, we give a detailed proof here (\refT{T1}).
In \refS{S2}, we correct also another error in \cite{Jakubowski}.
 
\begin{remark}
  We consider here functions defined on $\oi$. It is well-known that 
there is a version of the Skorohod topology for \cadlag{} functions on $\ooo$.
This too was extended by \cite{Jakubowski} to 
the space $D(\ooo,E)$ consisting of the \cadlag{} functions on $\ooo$ with
values in an 
arbitrary completely regular space $E$. 
Using the definition and methods of \cite[Section 4]{Jakubowski} together
with the proofs below, it is easy
to see that \refTs{T1} and \ref{T2} hold also for $D(\ooo,E)$.
\end{remark}

\begin{remark}
\citet[Theorem 1.3]{Jakubowski} is important also
in  the standard special case when $E$ is a metric space;
in this case it
shows that the Skorohod topology
does not depend on the choice of metric in $E$.
This important and useful fact seems to be largely ignored or at most
implicit in the literature rather than stated
explicitly. For example, it is a consequence of (and essentially equivalent to)
\cite[Problem 3.11.13]{EthierKurtz} or \cite[Exercise 16.5]{Kallenberg} 
(both given as exercises without proof, and
the latter stated only for complete separable metric spaces),
but also there it is not stated explicitly.
%Function $f:S\to T$. But only complete separable.
\end{remark}

\begin{ack}
  I thank Adam Jakubowski for helpful comments.
\end{ack}

\section{Preliminaries}\label{S2}
Let $E=(E,\tau)$ be a Hausdorff topological space.
Let $D(\oi,E)=D(\oi,E,\tau)$ 
be the space of all functions $x:\oi\to E$ that are 
\cadlag, i.e., 
right-continuous on $\oi$ and with left limits at all $t\in\ooi$.
(We usually omit $\tau$ from the notation.)
We denote the left limit at $t$ by
\begin{align}\label{ft-}
  f(t-):=\lim_{s\upto t} f(s),
\end{align}
for every function $f$ and $t\in\ooi$ such that this limit exists.
For completeness, we also define $f(0-):=f(0)$.

Let $\hI$ be the \emph{split interval} or \emph{arrow space}
defined by taking two points $t+$ and $t-$ for every $t\in\oi$ and letting
\begin{align}\label{hI}
  \hI:=\set{t-:t\in\ooi}\cup\set{t+:t\in\oi}.
\end{align}
\begin{remark}
  Note the asymmetry at the endpoints of our definition: $0-\notin\hI$. 
This is an unfortunate
consequence of the standard definition of $D(\oi)$ at the endpoints. For
other puposes one usually uses a symmetric version of $\hI$.
\end{remark}

We regard $\oi$ as a subset of $\hI$ by identifying $t+$ with $t$ for every
$t\in\oi$. 
We give $\hI$ the natural order, extending the standard order on $\oi$.
Formally, for $t,u\in\oi$,
\begin{align}
  t+<u+&\iff t<u,\\
  t-<u-&\iff t<u,\\
  t+<u-&\iff t<u,\\
  t-<u+&\iff t\le u.
\end{align}
We then give $\hI$ the order topology.
It is easily seen that a neighbourhood base at $t-$ is given by the
intervals (in $\hI$ with this order) $(s-,t-]$, $s<t$;
similarly
a neighbourhood base at $t+$ is given by the
intervals $[t+,u+)$, $u>t$, interpreted as $\set{1+}$ when $t=1$.
(Thus \set{1+} is isolated.)
It is well known, and easy to se, that $\hI$ is compact.

We regard $\oi$ as a subset of $\hI$ by identifying $t\in\oi$ with
$t+\in\hI$.
(Note that the subspace topology that this induces on $\oi$ regarded as a
subset of $\hI$ is \emph{not} the standard topology.)

\begin{remark}
  $\hI$ is totally disconnected, separable, and first countable, but not
  second countable and not metrizable, 
see e.g.\ \cite[Section 9.2]{SJ271}.
\end{remark}

If $f\in D(\oi,E)$, then $f$ has a natural extension to $\hI$ given
by \eqref{ft-}.

\begin{lemma}\label{LA}
  If\/ $E$ is a regular (T${}_3$) topological space, 
and $f:\oi\to E$ is a  function,
then the following are  equivalent:
\begin{romenumerate}
\item\label{LA1} 
$f$ is \cadlag{}. In other words, $f\in D(\oi,E)$.
\item\label{LA2} 
$f$ has a continuous extension to $\hI$.
\end{romenumerate}
If this holds, then the continuous extension to $\hI$ is unique, and is the
natural extension given by \eqref{ft-}.
\end{lemma}

\begin{proof}
\ref{LA2}$\implies$\ref{LA1}:
Suppose that $f$ has a continuous extension (also denoted by $f$) to $\hI$.

If $t\in\oio$ and $s_n\downto t$ in $\oi$ (with the usual topology), then
$s_n\to t$ also in $\hI$, and thus $f(s_n)\to f(t)$.

Similarly, 
if $t\in\ooi$ and $s_n\upto t$ in $\oi$, then
$s_n\to t-$ in $\hI$, and thus $f(s_n)\to f(t-)$.

Hence, regarded as a function on $\oi$, $f$ is \cadlag.

\ref{LA1}$\implies$\ref{LA2}:
We define an extension to $\hI$ by \eqref{ft-}. 
We claim that this extension is continuous on $\hI$.

First, let $t\in\oio$, and let $U$ be a neighbourhood of $f(t)$.
Since $E$ is regular, there exists a closed neighbourhood $V$ of $f(t)$ with
$V\subseteq U$.
Since $f$ is right-continuous by assumption, there exists $\eps>0$ such that
if $u\in[t,t+\eps]$, then $f(u+)=f(u)\in V$. Furthermore, $f$ has a left
limit $f(u-)$ at every such $u$, and since $V$ is closed, it follows that if
$u\in(t,t+\eps]$, then $f(u-)\in V$.
Hence, if $v\in\hI$ with $t\le v<t+\eps$, then $f(v)\in V\subseteq U$.
The neighbourhood $U$ was arbitrary, and thus $f$ is continuous at every
$t+\in\hI$.
(Recall that $[t,t+\eps)=(t-,t+\eps)$ is open in $\hI$.)

Similarly, $f$ is continuous at every $t-\in\hI$. Thus the extension $f$ is
continuous on $\hI$.

The continuous extension is unique, since $\oi$ is dense in $\hI$.
\end{proof}

\begin{remark}
  When $E=\bbR$ or $\bbC$, 
this extension gives an isomorphism $D\oi \cong C(\hI)$
(with these denoting spaces of real-valued or complex-valued functions,
respectively). 
$D\oi$, equipped with the supremum norm, is a Banach algebra, and 
its maximal ideal space can be identified with $\hI$; then
this isomorphism $D\oi \to C(\hI)$ is the Gelfand transform,
see e.g.\ \cite[Section 9.2]{SJ271}.
\end{remark}

\begin{corollary}\label{CA}
If\/ $E$ is a regular topological space and $f\in D(\oi,E)$, then
the set $\set{f(t):t\in\oi}\cup\set{f(t-):t\in\ooi}$  is a compact subset of
$E$.
In particular, the range $f(\oi)$ is relatively compact.
\end{corollary}
\begin{proof}
  $\hI$ is a compact space and the extension 
$f:\hI\to E$ is continuous by \refL{LA}.
Hence $f(\hI)$ is compact.
\end{proof}

\refC{CA} is \cite[Proposition 1]{Jakubowski}, but assumes that $E$ is regular.
It is erroneously claimed in \cite[Proposition 1]{Jakubowski} 
that it holds for Hausdorff spaces. 
The following example shows that this is incorrect;
\refL{LA} and \refC{CA} do not hold for arbitrary Hausdorff spaces $E$.
(Only the completely regular case, which is correct, 
is used later in \cite{Jakubowski}.)

Let $\tau_0$ be the standard topology on $\bbR$.
\begin{example}
  Let $\tau_K$ (often called the $K$-topology or the Smirnov topology)
be the topology on $\bbR$ defined by letting $K:=\set{\frac1n:n\in\bbN}$
and declaring a set $O\subseteq\bbR$ to be open if $O=U\setminus H$ where
$U$ is open in the usual topology $\tau_0$
and $\emptyset\subseteq H\subseteq K$.
Then the subspace topology on $\bbR\setminus\set0$ equals the standard
topology, but the topology at $0$ is different: a neigbourhood base
at 0 is given by the sets $(-\eps,\eps)\setminus K$.
The topology $\tau_K$ is Hausdorff (since it is finer than the standard
topology), but it is not regular since $K$ is a closed set and $0\notin K$,
but $0$ and $K$ cannot be separated by two open sets.
(See \eg{} 
\cite[Counterexample 64]{SteenSeebach}
or
\cite[Example 1.5.6, with a trivial  modification]{Engelking}.)

Note, for later use, that $K$ is closed but discrete and infinite,
and therefore not compact.

We let $\bbR_K=(R,\tau_K)$ denote $\bbR$ with this topology.

Define a function $f:\oi\to \bbR_K$ by
\begin{align}
  \begin{cases}
      f(0):=f(1):=0,
\\
f(t):=\frac12(t+\frac1n), & t\in[\tfrac1{n+1},\tfrac1n),\; n\in\bbN.
  \end{cases}
\end{align}
Note that $f(t)\notin K$ for all $t\in\oi$.

Evidently $f\in D(\oi,\bbR,\tau_0)$ (with the standard topology on $\bbR$), with
\begin{align}\label{b1}
  \begin{cases}
  f(x-)=f(x), & x\notin K,
\\
f(\frac1n-)=\tfrac1n,& n\in\bbN.  
  \end{cases}
  \end{align}
Thus $f(t)\in\bbR\setminus\set0$ for every $t\in(0,1)$ and 
$f(t-)\in\bbR\setminus\set0$ for every $t\in(0,1]$; since $\tau_K$ equals
the standard topology on $\bbR\setminus\set0$, 
it is easily seen that $f$ is
right-continuous and has left limits $f(t-)$
also in $\tau_K$ everywhere on $\ooi$.
Furthermore, since $f(t)\notin K$ for all $t$, it follows that $f$ is
(right-)continuous at $0$ too in $\tau_K$. 
Hence $f$ is \cadlag{} also for $\tau_K$, and thus $f\in D(\oi,\bbR_K)$.

This means that $f$ can be extended to $\hI$ using \eqref{b1}, also for
$\tau_K$. (This extension is necessarily the same as for the standard
topology on $\bbR).$
However, this extension of $f$ to $\hI$  is \emph{not} continuous,
since
$\frac1n-\to0$ in $\hI$ as \ntoo, 
but $f(\frac1n-)=\frac1n\not\to0$ in $\bbR_K$:
by construction, $U:=\bbR\setminus K$ is a neighbourhood of $0$ such that
$f(\frac1n-)\notin U$ for every $n$.
Hence, \refL{LA} does not hold for $\bbR_K$.

Similarly, \refC{CA} does not hold for $\bbR_K$ and the function $f\in
D(\oi,\bbR_K)$ above. In fact, the set 
$\set{f(t):t\in\oi}\cup\set{f(t-):t\in\ooi}= f(\hI)$ is not compact (and not
even relatively compact), since it contains
$\set{f(\frac1n-):n\in\bbN}=K$ as a closed but non-compact subset.
Moreover, $f(\hI)\subseteq \overline{f(\oi)}$ (for any \cadlag{} $f$), and
thus $\overline{f(\oi)}$ is not compact, i.e., $f(\oi)$ is not relatively
compact. 
\end{example}

\section{The Skorohod topology on $D(\oi,E)$}\label{S3}

Assume from now on that $E=(E,\tau)$ is a completely regular space.
%We write $D(\oi,E)$ as $D(\oi,E,\tau)$ when we want to specify the topology
%$\tau$ on $E$.

The Skorohod topology on $D(\oi,E)$ is defined by \cite{Jakubowski} 
(generalizing \cite{Mitoma})
as follows.
We use the fact that any completely regular topology
is generated by a family of pseudometrics $\set{d_i}_{i\in I}$ satisfying
\begin{gather}
  \label{ps1}
\forall_{a,b\in E}\exists_{i\in I} \quad d_i(a,b)>0,
\\
\forall_{i,j\in I}\exists_{k\in I} \quad\max(d_i,d_j)\le d_k.\label{ps2}
\end{gather}
More precisely, the topology on $E$ is generated by the functions
$d_i(a,\cdot):E\to\bbR$ with $i\in I$ and $a\in E$;
equivalently, the set of open balls for the pseudometrics $d_i$ forms a
base of the topology.
Furthermore, every pseudometric $d_i$ is continuous $E\times E\to\bbR$.

\begin{remark}
  Conversely, any such family of pseudometrics on a set defines a completely
regular topology. Note also that \eqref{ps2} is mainly for convenience, and
can be assumed without loss of generality, since for any family of
pseudometrics \set{d_i} satisfying \eqref{ps1}, we may add all finite maxima
%$d_{i_1}\lor\dots\lor d_{i_m}$ 
$\max(d_{i_1},\dots,d_{i_m})$
to the family; then \eqref{ps2} holds, and the
enlarged family defines the same topology as the original family.
(For the original family, the set of balls is a subbase for the topology.)
\end{remark}

To define the topology on $D(\oi,E)$, we choose such a family
$\set{d_i}_{i\in I}$.
For any pseudometric $d$ on $E$, we define %(as \cite{Jakubowski})
a corresponding pseudometric
$\td$ on $D(\oi,E)$ by
\begin{align}\label{td}
\td(x,y):=\inf_{\gl\in\gL}\max\Bigcpar{\sup_{t\in\oi}|\gl(t)-t|,\sup_{t\in\oi}
  d\bigpar{x(\gl(t)),y(t)}}  ,
\end{align}
where $\gL$ is the set of strictly increasing continuous functions $\gl$
mapping $\oi$ onto itself.
Finally, $D(\oi,E)$ is given the topology generated by the family of
pseudometrics $\set{\td_i}_{i\in I}$.

\begin{lemma}\label{LK}
Let the topology  on $E$ be defined by a family of pseudometrics
$\set{d_i}_{i\in I}$ satisfying \eqref{ps1}--\eqref{ps2}.
Let $\rho$ be any continuous pseudometric on $E$,
and suppose that $K$ is a compact subset of $E$. 
Then, for every $\eps>0$, there exists a pseudometric $d_i$ in the given
family and $\gd>0$, such that if $x\in K$ and $y\in E$ with $d_i(x,y)<\gd$,
then $\rho(x,y)<\eps$.
\end{lemma}
\begin{proof}
For every $z\in K$, the set $\set{y\in E:\rho(y,z)<\eps/2}$ is an open
neighbourhood of $z$, and thus there exists $j_z\in I$ and $\gd_z>0$ such that
\begin{align}\label{c1}
U_z:=  \set{y:d_{j_z}(y,z)<2\gd_z}
\subseteq\set{y\in E:\rho(y,z)<\eps/2}.
\end{align}
%Hence, if $x,y\in U_z$, then $\rho(x,y)<\eps$.

The open sets $U'_z:=  \set{y:d_{j_z}(y,z)<\gd_z}$
cover $K$, so we may select an finite subcover
$U'_{z_1},\dots,U'_{z_n}$. It follows from \eqref{ps2} that there exists 
$i\in I$ such that $d_{j_{z_k}}\le d_i$ for every $k=1,\dots,n$.

Let $\gd:=\min_{1\le k\le n} \gd_{z_{j_k}}$. If $x\in K$ and $y\in E$ with
$d_i(x,y)<\gd$, then choose $z_k$ such that $x\in U'_{z_k}$. We have
\begin{align}\label{c2}
%  d_{j_{z_k}}(x,y) \le d_i(x,y)<\gd
d_{j_{z_k}}(y,z) \le  d_{j_{z_k}}(x,y)  + d_{j_{z_k}}(x,z) 
\le d_i(x,y) + \gd_{z_k} <\gd + \gd_{z_k} \le 2\gd_{z_k};
\end{align}
thus $y\in U_{z_k}$. Furthermore, $x\in U'_{z_k}\subseteq U_{z_k}$.
Consequently, \eqref{c1} shows that $\rho(x,z_k)<\eps/2$ and
$\rho(y,z_k)<\eps/2$, and thus $\rho(x,y)<\eps$.
\end{proof}

\begin{theorem}[Jakubowski \cite{Jakubowski}]\label{T1}
  Let $\set{d_i}_{i\in I}$ and $\set{\zeta_j}_{j\in J}$ be two families of
  pseudometrics on $E$ satisfying \eqref{ps1} and \eqref{ps2}.
Let the topology $\tau$ generated by $\set{d_i}_{i\in I}$ be coarser than
the topology $\gs$ generated by  $\set{\zeta_j}_{j\in J}$.
Then obviously $D(\oi,E,\tau)\supseteq D(\oi,E,\gs)$
and the topology on $D(\oi,E,\gs)$ generated by the pseudometrics
$\set{\tzeta_j}_{j\in J}$ is finer than the topology induced by
$D(\oi,E,\tau)$ generated by the pseudometrics $\set{\td_i}_{i\in I}$.

In particular, if the families $\set{d_i}_{i\in I}$ and $\set{\zeta_j}_{j\in J}$
define the same topology on $E$, then the 
families $\set{\td_i}_{i\in I}$ and $\set{\tzeta_j}_{j\in J}$
define the same topology on $D(\oi,E)$.
\end{theorem}
\begin{proof}
Suppose that $x\in D(\oi,E,\gs)$. Then obviously $x\in D(\oi,E,\tau)$.

Let $i\in I$. 
Then the pseudometric $d_i$ is continuous on $(E,\tau)$, 
and thus also on $(E,\gs)$.
By \refC{CA}, the range $x(\oi)$ is relatively compact in $(E,\gs)$, and
is thus a subset of some compact set $K$ in $(E,\gs)$.

Let $\eps>0$.
\refL{LK} (applied to $(E,\gs)$, $\set{\zeta_j}_{j\in J}$, and $d_i$) 
shows that there exists $j\in J$
and $\gd>0$ such that if 
$y\in D(\oi,E,\gs)$,
$t\in\oi$, $\gl\in\gL$, 
and $\zeta_j(x(\gl(t)),y(t))<\gd$, then $d_i(x(\gl(t)),y(t))<\eps$.
Consequently, recalling \eqref{td},
if $\tzeta_j(x,y)<\min(\gd,\eps)$, then $\td_i(x,y)\le\eps$.
It follows that $\td_i$ is continuous 
on $D(\oi,E,\gs)$ for the topology defined by
$\set{\tzeta_j}_{j\in J}$, and the result follows.
\end{proof}

This theorem shows that we can unambiguously talk about $D(\oi,E)$ as a
topological space, for any completely regular topological space $E$.  

The theorem 
has the following corollary.

\begin{theorem}\label{T2}
Suppose that $E$ and $F$ are two completely regular spaces, and that
$\psi:E\to F$ is a continuous function.
Define $\Psi: D(\oi,E)\to D(\oi,F) $ by 
\begin{align}\label{Psi}
\Psi(x):=\Psi\circ x,
\qquad \text{i.e.,} \qquad 
\Psi(x)(t):=\Psi\xpar{x(t)}, \quad t\in\oi.
\end{align}
Then $\Psi$ is continuous $D(\oi,E)\to D(\oi,F)$.
\end{theorem}
\begin{proof}
  It is clear that $\Psi$ maps $D(\oi,E)$ into $D(\oi,F)$.

Let the topologies of $E$ and $F$ be defined by families of pseudometrics
$\set{d_i}_{i\in I}$ and $\set{\zeta_j}_{j\in J}$, respectively.
For $j\in J$, define $\xzeta_j(a,b):=\zeta_j(\psi(a),\psi(b))$ for $a,b\in E$;
then $\xzeta_j$ is a continuous pseudometric on $E$.
Hence, for every $i\in I$ and $j\in J$, $\max(d_i,\xzeta_j)$ is a continuous
pseudometric on $E$. It follows that 
$\set{d_i}_{i\in I}\cup\set{\xzeta_j}_{j\in J}\cup
\set{\max(d_i,\xzeta_j)}_{(i,j)\in I\times J}$ is a family of pseudometrics
on $E$ that satisfies \eqref{ps1}--\eqref{ps2}, and that this family defines
the same topology on $E$ as $\set{d_i}_{i\in I}$.

By \refT{T1}, the corresponding pseudometrics on $D(\oi,E)$ generate the
topology on $D(\oi,E)$. In particular, every $\txzeta_j$ is a continuous
semimetric on $D(\oi,E)$. It follows from the definitions \eqref{td}
and \eqref{Psi} that for any $x,y\in D(\oi,E)$,
\begin{align}
  \tzeta_j\bigpar{\Psi(x),\Psi(y)}
=\txzeta_j(x,y).
\end{align}
We have shown that this is a continuous function of $(x,y)\in D(\oi,E)^2$.
Since the pseudometrics $\tzeta_j$ generate the topology on $D(\oi,F)$, it
follows that $\Psi$ is continuous.
\end{proof}

%\appendix

%\section*{Acknowledgement}
%This work  was partially supported by 
%a grant from the Knut and Alice Wallenberg Foundation.

\newcommand\AMS{Amer. Math. Soc.}
\newcommand\Springer{Springer-Verlag}
\newcommand\Wiley{Wiley}

\newcommand\vol{\textbf}
\newcommand\jour{\emph}
\newcommand\book{\emph}
\newcommand\inbook{\emph}
\def\no#1#2,{\unskip#2, no. #1,} %(typeset after year) 
\newcommand\toappear{\unskip, to appear}

\newcommand\arxiv[1]{\texttt{arXiv}:#1}
\newcommand\arXiv{\arxiv}

\newcommand\xand{and }
\renewcommand\xand{\& }

\def\nobibitem#1\par{}


\begin{thebibliography}{99}

\bibitem[Billingsley(1999)]{Billingsley2ed}
Patrick Billingsley.
\book{Convergence of Probability Measures}.
2nd ed., 
Wiley, New York, 1999.


\bibitem{Engelking}
Ryszard Engelking.
\emph{General Topology}. 
%Translated from the Polish by the author. 
2nd ed.,
%Sigma Series in Pure Mathematics, 6. 
Heldermann Verlag, Berlin, 1989. 
%viii+529 pp. ISBN: 3-88538-006-4

\bibitem{EthierKurtz}
Stewart N. Ethier and  Thomas G. Kurtz.
\emph{Markov Processes. Characterization and Convergence.} 
%Wiley Series in Probability and Mathematical Statistics: Probability and
%Mathematical Statistics. 
John Wiley \& Sons, Inc., New York, 1986. 
%x+534 pp. ISBN: 0-471-08186-8
%\MR{0838085}

\bibitem[Jakubowski(1986)]{Jakubowski}
Adam Jakubowski. 
On the Skorokhod topology. 
\emph{Ann. Inst. H. Poincaré Probab. Statist.} \vol{22}:3 (1986), 263--285.

\bibitem{SJ271}
Svante Janson and  Sten Kaijser.
Higher moments of Banach space valued random variables.
\emph{Mem. Amer. Math. Soc.} \vol{238} (2015), no. 1127. 

\bibitem{Kallenberg}
Olav Kallenberg.
\book{Foundations of Modern Probability.}
2nd ed., Springer, New York, 2002. 

\bibitem[Mitoma(1983)]{Mitoma}
Itaru Mitoma. 
Tightness of probabilities on $C([0,1];S')$ and $D([0,1];S')$. 
\emph{Ann. Probab.} \vol{11} (1983), no. 4, 989--999. 

\bibitem[Skorohod(1956)]{Skorohod}
A. V. Skorohod.
Limit theorems for stochastic processes.
\emph{Teor. Veroyatnost. i Primenen.} \vol1 (1956), 289--319. (Russian.)
English transl.:
\emph{Theory Probab. Appl.} \vol1 (1956), 261--290.

\bibitem{SteenSeebach}
Lynn A.  Steen and  J. Arthur  Seebach, Jr. 
\emph{Counterexamples in Topology}. 
Holt, Rinehart and Winston, Inc., New York--Montreal--London, 1970. 

\end{thebibliography}
\end{document}